\documentclass[12pt]{amsart}
\usepackage{mathrsfs}
\usepackage{a4wide}
\usepackage{amssymb}
\usepackage{amsthm}
\usepackage{amsmath}
\textheight8.7in \textwidth6.7in \numberwithin{equation}{section}
\theoremstyle{plain}

\newtheorem{theorem}{Theorem}
\newtheorem{lemma}[theorem]{Lemma}
\newtheorem{corollary}[theorem]{Corollary}

\newtheorem{definition}[theorem]{Definition}

\newcommand{\N}{\mathbb{N}}

\newcommand{\R}{\mathbb{R}}

\newcommand{\eps}{\epsilon}

\newcommand{\inv}{^{-1}}
\newcommand{\vp}{\varphi}
\newcommand{\F}{\mathscr{F}}
\newcommand{\G}{\mathscr{G}}
\newcommand{\floor}[1]{\left\lfloor #1\right\rfloor}
\newcommand{\Pf}{\mathscr{P}_f}
\newcommand{\Pa}{\mathscr{P}_{q,a}}
\newcommand{\A}{\mathscr{A}}
\newcommand{\del}{\delta}
\newcommand{\B}{\mathcal{B}}
\newcommand{\ve}{\varepsilon}
\newcommand{\spmod}{\hspace{-8pt}\pmod}

\theoremstyle{remark}

\newtheorem*{remark}{Remark}
\newtheorem*{example}{Example}
\numberwithin{theorem}{section} \numberwithin{equation}{section}

\begin{document}
\title[Strings of Special Primes in Arithmetic Progressions] {Strings of Special Primes in Arithmetic Progressions}

\author{Keenan Monks} 

\author{Sarah Peluse} 

\author{Lynnelle Ye} 

\address{Keenan Monks, 73 N James St, Hazleton, PA 18201}
\email{monks@harvard.edu}
\address{Sarah Peluse, 491 Parkview Terrace, Buffalo Grove, IL 60089}
\email{peluse@uchicago.edu}
\address{Lynnelle Ye, P.O. Box 16820, Stanford, CA 94309}
\email{lynnelle@stanford.edu}
\thanks{The authors are grateful to the NSF's support of the REU at Emory University.} 

\begin{abstract}
The Green-Tao Theorem, one of the most celebrated theorems in modern number theory, states that there exist arbitrarily long arithmetic progressions of prime numbers. In a related but different direction, a recent theorem of Shiu proves that there exist arbitrarily long strings of consecutive primes that lie in any arithmetic progression that contains infinitely many primes. Using the techniques of Shiu and Maier, this paper generalizes Shiu's Theorem to certain subsets of the primes such as primes of the form $\floor{\pi n}$ and some of arithmetic density zero such as primes of the form $\floor{n\log\log n}$.
\end{abstract}
\maketitle

\section{Introduction and statement of results}
In 1837, Dirichlet famously proved that every arithmetic progression of the form $\{a+qn\}_{n\in\N}$ where $(q,a)=1$ contains infinitely many primes.  This was the first major accomplishment in the direction of relating the primes to arithmetic progressions, a problem that has generated interest for hundreds of years. As early as 1770, Waring and Lagrange investigated the size of common differences in arithmetic progressions of primes, and in 1939 van der Corput \cite{Van} proved that there are infinitely many arithmetic progressions of length three in the primes. 

More recently, Gowers \cite{Gowers} used harmonic analysis to reprove Szemer\'edi's Theorem regarding arithmetic progressions in subsets of the integers. Green and Tao \cite{Green} then generalized his methods (see also the expository article by Kra \cite{Kra}) to prove that there are arbitrarily long arithmetic progressions within the primes; i.e.~for any positive integer $k$, there exist $a,q$ so that $$\{a+q,a+2q,\dotsc,a+kq\}$$ are all prime. 

Here we consider another natural question: are there arbitrarily long strings of consecutive primes all of which lie in the same arithmetic progression? That is, given a positive integer $k$ and $a,q$ with $(q,a)=1$, if $p_n$ is the $n$th prime, can we find $n$ so that $$p_{n+1}\equiv p_{n+2}\equiv\dotsb\equiv p_{n+k}\equiv a\spmod{q}?$$ 

In a remarkable paper of Shiu \cite{Shiu}, the question is answered. By adapting the work of Maier\footnotemark  (see \cite{Maier},~\cite{Gran}), Shiu is able to obtain a favorable distribution of primes in any arithmetic progression among the rows of a ``Maier matrix,'' whose rows are relatively small intervals of integers. Using this, he obtains a lower bound on the length of the longest such string of primes within some row of the matrix, a bound which approaches infinity as the matrix grows.

\footnotetext{This was generalized by Thorne to function fields in his Ph.D thesis \cite{Thorne}.}

In this paper, we exploit Shiu's novel method to prove analogous results about well-distributed subsets of the primes. In particular, we prove the following:
\begin{theorem}\label{Main}
If the set of primes $\A$ is {\bf \emph{well-distributed}} (see Section~\ref{prelims} for the definition,) then there exist arbitrarily long strings of consecutive primes in $\A$ all of which lie in any given progression $\{a+qn\}$ with $(q,a)=1$.
\end{theorem}
For an example of what Theorem 1.1 implies, consider the set of $\pi$-Beatty primes 
\begin{eqnarray*}
\B_{\pi}&=&\{p\text{ prime}: p = \floor{\pi},\floor{2\pi},\floor{3\pi},\dotsc\}\\
&=&\{3,31,37,43,47,53,59,97,\cdots\}.
\end{eqnarray*} Write $\B_{\pi}=\{p_{\B_{\pi}}(1)<p_{\B_{\pi}}(2)<\dotsb\}$, where $p_{\B_{\pi}}(n)$ is the $n$th smallest prime in $\B_{\pi}$. Then for any length $k$, modulus $q$ and remainder $a$ where $(q,a)=1$, we can find some $n$ for which
\[
p_{\B_{\pi}}(n+1)\equiv p_{\B_{\pi}}(n+2)\equiv\dotsb\equiv p_{\B_{\pi}}(n+k)\equiv a\spmod q.
\]
For example, the first string of six consecutive primes in $\B_{\pi}$ all congruent to $5\pmod{7}$ are
$$26402437, 26402507, 26402591, 26402843, 26402899, 26402927.$$

We can generalize this result as follows for irrational numbers of \textit{finite type}. An irrational number $\alpha$ is of finite type if 
\[
\sup\{r\in\mathbb{R}:\liminf_{n\in\N}n^r\|n\alpha\|=0\}
\]
is finite where $\|r\|=\min_{n\in\mathbb{Z}}|r-n|.$ Almost every real number is of finite type (\cite{Khin}), in particular all of the irrational real algebraic numbers (\cite{Roth1},\cite{Roth2}).

\begin{corollary}
\label{beatty}
For any irrational $\alpha>0$ of finite type, there exist arbitrarily long strings of consecutive primes in $\B_{\alpha}=\{\floor{\alpha},\floor{2\alpha},\floor{3\alpha},\cdots\}$ all of which lie in any progression $\{a+qn\}$ with $(q,a)=1$.
\end{corollary}

It does not come as a surprise that the set $\B_{\alpha}$ satisfies the conditions of being well-distributed, since these have positive density in the natural numbers and are very uniform. However, we can improve upon this and use the work of Leitmann~\cite{Leit} to obtain sets of primes with arithmetic density zero in the primes that satisfy our distribution conditions. Leitmann considers general sets of the form $\Pf=\{\floor{f(n)}\}_{n\in \N}$ for nice functions $f$ (the set of such functions is denoted $\F$), and proves the following analogue of the Prime Number Theorem for arithmetic progressions: 
\[
\frac{\vp(X)}{\phi(q)\log X}\ll\pi_{\Pf}(X;q,a)\ll \frac{\vp(X)}{\phi(q)\log X},
\]
where $\vp(x)=f^{-1}(x)$. We define $\G$ as a certain nice subset of functions $g(x)$ such that $xg(x)\in\F$. Using this, we can prove that these sets also contain strings of primes in arithmetic progressions. This is summarized in the following theorem.
\begin{theorem}
\label{leitsimple}
Let $f(x)=xg(x)$ where $g\in\G$. Then there exist arbitrarily long strings of consecutive primes in $\Pf$ all of which lie in any given progression $\{a+qn\}$ with $(q,a)=1$.
\end{theorem}
\begin{example}
The function $f(x)=x\log\log x$ works since $\log\log x\in \G$. That is, suppose we order the set $\{p\text{ prime}: p=\floor{n\log\log n}, n=1,2,\dotsc\}$ in the form $p_f(1)<p_f(2)<\dotsb$. Then for any length $k$, modulus $q$ and remainder $a$ where $(q,a)=1$, we can find some $n$ for which
\[
p_f(n+1)\equiv p_f(n+2)\equiv\cdots\equiv p_f(n+k)\equiv a \spmod q.
\]
\end{example}

The proofs of these results, and in particular Theorem \ref{Main}, generalize Shiu's work in \cite{Shiu}. Moreover, we, like Shiu, produce in Theorem \ref{main} a rigorous upper bound for when a string of length $k$ must appear. Theorem \ref{Main} is simply a consequence of the more precise Theorem \ref{main}. 

To find a long string of primes in one of our well-distributed sets, we study the object known as the Maier matrix. In general, the matrix is of the form
\[
M=\left(\begin{array}{cccc}
Q+1& Q+2 & \cdots & Q+{yz}\\
2Q+1 &2Q+2&\cdots& 2Q+{yz}\\
\vdots&\vdots&\ddots&\vdots\\
Q^D+1&Q^D+2&\cdots&Q^D+{yz}
\end{array}\right),
\]
where $Q$ is a carefully chosen product of primes and $yz$ is another nicely chosen quantity much less than $Q.$ The rows of $M$ are intervals of integers, and it is in one of these intervals that we seek to find our string of $k$ consecutive primes congruent to $a\pmod{q}.$  Similarly, the columns of $M$ are arithmetic progressions with common difference $Q.$ Here the length of the rows, $yz,$ is extremely small with respect to $Q$, so we are unable to use standard distribution results  like the Prime Number Theorem to get a long enough string.  Our method of proof is to show that there are a large number of primes equivalent to $a\pmod{q}$ versus those that are not equivalent to $a\pmod{q}$ in $M.$ Then, using a combinatorial argument, we use this to derive a lower bound on the length of a string of consecutive primes equivalent to $a\pmod{q}$ in the rows of $M$ and show that this bound goes to infinity.

To count the number of primes congruent to $a\pmod q$ versus those that are not congruent to $a\pmod q$ in $M,$ we utilize the fact that our set of primes is well-distributed. This gives us the existence of a Prime Number Theorem for arithmetic progressions, which lets us count primes via the columns. The only columns which can possibly contain primes are those with first term $Q+i$ where $(Q,i)=1,$ so these will be the only columns that concern us. We pick $Q$ such that $q|Q,$ which means that we can tell whether a number is equivalent to $a\pmod{q}$ based on which column it lies in. To be precise, $n\equiv a\pmod{q}$ and $n\in M$ if and only if $n$ lies in a column with first term $Q+i$ where $i\equiv a\pmod{q}.$ Define ``good columns'' to be the columns of $M$ containing primes whose elements are equivalent to $a\pmod{q},$ and similarly define ``bad columns'' to be the columns containing primes whose elements are not equivalent to $a\pmod{q}.$ We have for our set $\A,$
\[
\frac{1}{\phi(Q)}E(X)\ll\pi_\A(X;Q,l)\ll \frac{1}{\phi(Q)}{E(X)}
\]
for some function $E(X)$ whenever $(q,l)=1.$ What is important here is that we can use the above asymptotic on the columns of $M$ to reduce the problem of estimating the number of good primes versus bad in $M$ to estimating the number of good columns versus bad. We can rig our choice of $Q$ such that $\#S$ is much larger than $\#T,$ so the result follows.

In Section 2, we will provide the definitions and technical lemmas necessary to proceed with the proofs of the main theorems. In Section 3, we prove the main result (Theorem~\ref{Main}) and Corollary~\ref{beatty}. Finally, we adapt the work of Leitmann in Section 4 to show that our main result applies to the class of sequences discussed in Theorem~\ref{leitsimple}.

\section*{Acknowledgements}

We would like to thank Robert Lemke Oliver for his guidance throughout this project, and Ken Ono for his advice and encouragement.

\section{Preliminaries}
\label{prelims}
Throughout this paper, $\phi$ will be Euler's totient function and $f(X)\ll g(X)$ will mean $f(X)=O(g(X)).$ Given $\A$ a subset of the primes and $(q,a)=1,$ write
\[
\pi_\A(X;q,a)=\#\{p\in\A:p\leq X, p\equiv a \spmod{q}\}
\]
and $p_\A(n)$ for the $n^{th}$ smallest prime in $\A,$ so that
\begin{align}
\A=\{p_\A(1)<p_\A(2)<\dots\}.
\end{align}
A string of $k$ consecutive primes in $\A$ refers to a set of the form $\{p_\A(n),p_\A(n+1),\dots,p_\A(n+k-1)\}$ where $n\in\mathbb{N}.$

Given that the theory of Dirichlet $L$-functions is integral to the study of primes in arithmetic progressions, it follows that they should be involved somehow in the proof of Theorem \ref{main} and its corollaries. However, we will not work with $L$-functions directly in this paper, though we use several results whose proofs are dependent on their properties. For the remainder of this paper, $C$ will refer to some fixed constant whose existence comes from Lemma 1 of Shiu \cite{Shiu}. $C$ is such that, for all $q\in\N$ and large enough $X,$ there exists a $y$ and prime $p_0\gg\log y$ such that none of the $L$-functions modulo $q\prod_{p\leq y, p\neq p_0}p$ have a zero in the region defined by Equation~\ref{region} below.

\begin{definition}
We say that a modulus $q$ is exceptional if there is an $L$-function induced by a character modulo $q$ with a zero in the region
\begin{align}
\label{region}
1\ge\mathscr{R}(s)>1-\frac{C}{\log(q(|\mathscr{I}(s)|+1))}.
\end{align}
\end{definition}
\begin{definition}
We say that a subset of the primes $\A$ is well-distributed if there exists a function $D(X)$ satisfying
\begin{align}
\log D(X)=o\left(\frac{\log\log X\log\log\log\log X}{\log\log\log X}\right)
\end{align}
such that for all $q$ not exceptional and $X\geq q^{D(X)},$ the estimate
\begin{align}
\pi_\A(X;q,a)\asymp \frac{1}{\phi(q)}E(X)
\end{align}
holds uniformly for $(a,q)=1$ for some function $E(X)$ with
\begin{align}
\log\left(\frac{X}{\log X E(X)}\right)=o\left(\frac{\log\log X\log\log\log\log X}{\log\log\log X}\right).
\end{align}
\end{definition}
Note that if $\A$ is taken to be the set of all primes as in Shiu's theorem, we  have that $E(X)=\frac{X}{\log X}$, and in general we want $\A$ to be large enough that $\frac{X}{\log X E(X)}$ grows slowly. The upper bound on $D(X)$ is to ensure that we are able to choose $q$ not too much smaller than $X$, as we will need to do. Shiu's theorem, for example, takes $D(X)$ to be constant.

Adopting the notation of Shiu \cite{Shiu}, we make the following definition.
\begin{definition}
Given $q\in\mathbb{N},$ $q\geq 2,$ define
\[
A_+=\{a\in\mathbb{N}:p|q\Rightarrow a\equiv 1\spmod{p}\},
\]
\[
A_-=\{a\in\mathbb{N}:p|q\Rightarrow a\equiv -1\spmod{p}\},
\]
and $A_\pm=A_+\cup A_-.$
\end{definition}
Though the existence of arbitrarily long strings of primes congruent to $a\pmod{q}$ in well-distributed subsets of the primes does not depend on the choice of $a$ as long as $(q,a)=1,$ we are able to get better bounds on when these strings must appear in the case that $a\in A_\pm.$
For $a\in \mathbb{N}$, we write
\begin{align}
\log(y,t)=
\begin{cases}
\log y & \text{if }a\in A_\pm\\
\log t & \text{if }a\notin A_\pm
\end{cases}.
\end{align}
Finally, we record two lemmas that will be useful later.
\begin{definition}
For any $q$, let $\mathscr{S}_q(z)$ denote the set of positive integers $n\le z$ with $n$ a product of primes congruent to $1\pmod{q}$. Where this is not ambiguous, we will simply write $\mathscr{S}(z)$.
\end{definition}

\begin{lemma}[Shiu~\cite{Shiu}]
\label{scriptS}
We have
\begin{align}
\# \mathscr{S}_q(z)\asymp\frac{z(\log z)^{1/\phi(q)}}{\log z}
\end{align}
where the implied constant may depend on $q$.
\end{lemma}

The following classical theorem estimates the distribution of $t$-smooth numbers, or integers whose prime divisors are all less than $t$. 

\begin{lemma}[De Bruijn~\cite{DeBr}]
\label{smooth}
Let $\Psi(x,t)$ denote the number of positive integers $n\le x$ which are products of primes less than $t$, where $t\le x$ and $t\rightarrow\infty$ as $x\rightarrow\infty$. Then if $u=\log x/\log t$, we have
\begin{align}
\Psi(x,t)\le x(\log t)^2\exp(-u\log u-u\log\log u+O(u)).
\end{align}
\end{lemma}

\section{Main Result}
We now prove our main result, Theorem \ref{Main}, restated below as Theorem \ref{main}. Here we give the precise bounds on the length of a string of consecutive primes in the set $\A$ which are in an arithmetic progression, and it is clear that these go to infinity.
\begin{theorem}
\label{main}
Let $\A$ be a well-distributed subset of the primes, $(q,a)=1,$ and set $F(X)=\frac{X}{E(X)\log X}.$ Then there exists a string of length $k$ of consecutive primes in $\A$ all congruent to $a\bmod{q}$, all of which are less than $X,$ and where, if $a\in A_\pm,$
\[
k\gg\left(\frac{\log\log X}{\log(\max\{D(X),F(X)\})}\right)^{1/\phi(Q)}
\]
and otherwise,
\[
k\gg \left( \frac{\log\log X\log\log\log\log X}{\log\log\log X\log( \max\{D(X),F(X)\})} \right)^{1/\phi(Q)}.
\]
\end{theorem}
\subsection{Proof of Theorem \ref{main}}
The proof of this theorem has the same general structure of Shiu's beautiful proof in \cite{Shiu} of the existence of arbitrarily long strings of consecutive primes lying in any arithmetic progression that can contain primes. 

First we construct a number $Q$ which gives the Maier matrix the desired distribution properties. Given $y\in\mathbb{N}$ and $p_0$ some prime, we will write $P_q(y,p_0)=q\prod_{p_0\neq p\leq y}p.$ By Lemma 1 in \cite{Shiu}, we can choose $y\gg \frac{\log X}{D(X)}$ and a prime $p_0\gg\log y$ such that
\[
X^{\frac{1}{3D(X)}}\leq P_q(y,p_0)\ll X^{\frac{1}{3D(X)}}\left(\frac{\log X}{D(X)}\right)^2
\]
and $P_q(y,p_0)$ is not exceptional.
Now, define
\begin{align}
\label{Qdivs}
\Pa=
\begin{cases}
\{p\leq y :  p\not\equiv 1\spmod{q}\}\setminus\{p:p= p_0 \text{ or } p|q\} & \mbox{if } a\in A_\pm \\
\{p\leq y: p\not\equiv 1,a\spmod{q}\}\\
\cup\{t\leq p\leq y: p\equiv 1\spmod{q}\}\\
\cup\{p\leq\frac{yz}{t}: p\equiv a\spmod{q}\}\setminus\{p:p=p_0\text{ or }p|q\} & \mbox{otherwise,}
\end{cases}
\end{align}
where $z=\max\left\{\left(\frac{X}{\log X E(X)}\right)^3,D(X)^3\right\}$,  $t=\exp\left(\frac{\frac{1}{4}\log y\log\log\log y}{\log\log y}\right),$ and
\[
Q_q(y,p_0)=q\prod_{p\in\Pa}p.
\]
Note that $Q_q(y,p_0)|P_q(y,p_0),$ so $Q_q(y,p_0)$ is also not exceptional. For the remainder of this proof, we will use the shorthand $Q$ for $Q_q(y,p_0).$

Since $\frac{X D(X)}{\log X E(X)}=o(z)$ and
\[
\log z=o\left(\frac{\log\log X\log\log\log\log X}{\log\log\log X}\right)
\]
by our choice of $z,$ $\log z=o(\log(y,t)).$ In addition, $z<y,$ $t\leq\sqrt{y}\leq\sqrt{yz},$ so $t\leq\frac{yz}{t}.$ Suppose that $a\notin A_\pm.$ Then $\Pa$ contains the set
\[
\{p\leq t : p\equiv a\spmod{q}\}\cup\{t\leq p\leq y: p\equiv 1\spmod{q}\}.
\]
Hence, the Prime Number Theorem for arithmetic progressions yields
\[
\sum_{p<y, p|Q}\log p=\left(1-\frac{1}{\phi(q)}\right)\sum_{p<y}\log p + o(y).
\]
Because $\frac{1}{\phi(q)}\leq\frac{1}{2}$ whenever $q>2,$ we have $\log Q>\frac{1}{3}\log P_q(y,p_0)$. Hence, $Q$ is on the order of $X^{1/D(X)}.$ Since we assumed that $\A$ is well-distributed, this implies, for any $l,$ that
\[
\pi_\mathscr{A}(X;Q,l)\asymp \frac{1}{\phi(Q)}E(X).
\]
As was outlined in the introduction, this sort of estimate will be necessary to count the number of primes in our Maier matrix equivalent to $a\pmod{q}.$ 

Now, in order to construct our Maier matrix, we make the following definition. First, pick natural numbers $m$ and $n$ such that $m,n\equiv 0\pmod{p}$ for all $p|\frac{Q}{q}$ and $m\equiv a-1 \pmod{q}$, $n\equiv a+1 \pmod{q}.$ Then set
\begin{align}
I=
\begin{cases}
\{m+1,m+2,\dots,m+yz\} & \mbox{if }a\in A_+ \\
\{n-yz+1,n-yz+2,\dots,n\} & \mbox{if }a\in A_-\\
\{1,2,\dots,yz\} & \mbox{if } a\notin A_\pm
\end{cases}.
\end{align}
Write $I=\{i_1<i_2<\dots<i_{yz}\}$. We define the Maier matrix $M$ to be
\begin{align}
\begin{pmatrix}
Q+i_1 & Q+i_2 & \dots & Q+i_{yz} \\
2Q+i_1 & 2Q+i_2 & \dots & 2Q+i_{yz} \\
\vdots & \vdots & \ddots & \vdots \\
Q^D+i_1 & Q^D+i_2 & \dots & Q^D+i_{yz} \\
\end{pmatrix}.
\end{align}
Note that $M$ is a matrix with $Q^{D-1}$ rows and $yz$ columns, whose columns form arithmetic progressions with common difference $Q$ and whose rows are intervals of integers.  In one of these intervals, we will be guaranteed to find a string of length $k$.

 In order to show this, we will show that there are so many primes equivalent to $a\pmod{q}$ in $M$ in relation to those that are not that some row is forced to have a string of length $k$. In this vein, we make the following definitions related to $M.$ Define
\begin{align}
S= \{i\in I : (i,Q)=1, i\equiv a \spmod{q}\}
\end{align}
and
\begin{align}
T=\{i\in I : (i,Q)=1, i\not\equiv a \spmod{q}\}.
\end{align}

These sets consist of elements of the first row of $M$ with $Q$ subtracted off of each term. Hence, $S$ and $T$ taken together represent all of the first terms of arithmetic progressions in the columns of $M$ that could possibly contain primes. $S$ represents columns of elements that are equivalent to $a\pmod{q}$ which could contain primes and $T$ represents columns of elements that are not equivalent to $a\pmod{q}$ which could contain primes. 

Our goal is to count the number of elements in $S$ versus $T$ and combine this with our asymptotic expression for $\pi_\A(X;Q,l)$ on the columns of $M$ to bound the number of primes congruent to $a\pmod{q}$ in $M\cap\A.$ We define the sets
\begin{align}
P_1=\{p\in M\cap\A : p\equiv a\spmod{q}\}
\end{align}
and
\begin{align}
P_2=\{p\in M\cap\A : p\not\equiv a\spmod{q}\}.
\end{align}
For the sake of brevity, we will refer to primes in $P_1$ as \emph{good primes} and primes in $P_2$ as \emph{bad primes}. Finally, write $M'$ to mean the rows of $M$ which contain bad primes.

There are two cases for the distribution of good primes versus bad in $M.$ Either there is a large ratio of good primes to bad primes in some row in $M'$ or there is a large number of good primes in $M\setminus M'$. Precisely, either there exists an interval $R_0\in M'$ such that
\[
 \#(R_0\cap P_1)\geq\frac{\# P_1}{2\# P_2}\#(R_0\cap P_2)
\]
or
\[
\#(P_1\cap(M\setminus M'))\geq\frac{1}{2}\#P_1.
\]
To see this, suppose to the contrary that neither of the above hold. Then we would have
\begin{eqnarray*}
\#P_1&=&\#(P_1\cap M')+\#(P_1\cap(M\setminus M'))\\
&=&\sum_{R\in M'}\#(P_1\cap R)+\#(P_1\cap (M\setminus M'))\\
&<&\frac{\#P_1}{2\#(P_2)}\sum_{R\in M'}\#(P_2\cap R)+\frac{\#P_1}{2}\\
&=&\frac{\#P_1}{2\#P_2}\#P_2+\frac{\#P_1}{2}\\
&=&\#P_1,
\end{eqnarray*}
which is clearly a contradiction.

We now break down the rest of the proof into these two cases.

\textit{Case I.} Suppose that there exists a row $R_0\in M'$ containing $\frac{\#P_1}{2\#P_2}$ times as many good primes as bad. By the Pigeonhole Principle, this row must contain a string of $k$ consecutive good primes with $k\gg\frac{\#P_1}{\#P_2}.$ Furthermore, since we know
\[
\pi_\A(X;Q,l)\asymp\frac{1}{\phi(Q)}E(X),
\]
we can write
\[
\#P_1\gg\#S\frac{1}{\phi(Q)}E(X)\ \ \ \  \mbox{ and } \ \ \ \  \#P_2\ll\#T\frac{1}{\phi(Q)}E(X).
\]
This implies that $k\gg\frac{\#S}{\#T}.$

\textit{Case II.} Suppose that $\#(P_1\cap(M\setminus M'))\geq\frac{1}{2}\#P_1$. Note that the total number of rows in $M$ is $Q^{D(X)-1}=\frac{Q^{D(X)}}{Q}\ll \frac{X}{Q}.$ So, clearly the number of rows in $M\setminus M'$ is $\ll \frac{X}{Q}.$ In the worst case, the primes in $P_1$ are distributed evenly amongst intervals in $M\setminus M'.$ Hence, since there are no bad primes in $M\setminus M'$, one of the rows in $M\setminus M'$ contains a string of length $k\gg \frac{Q\#P_1}{X}$ consecutive good primes.

Recall that we have 
\[
\#P_1 \gg \#S\frac{1}{\phi(Q)}E(X),
\]
so that
\[
k\gg\#S\frac{Q}{X\phi(Q)}E(X).
\]
We can estimate $\frac{1}{\phi(Q)},$ using the multiplicativity of $\phi.$ We have
\[
\frac{Q}{\phi(Q)}=\frac{q}{\phi(q)}\prod_{p\in\Pa}\left(1-\frac{1}{p}\right)^{-1},
\]
and by a generalization of Mertens' Theorem (Theorem 429 in \cite{Hardy}), this yields
\[
\frac{Q}{\phi(Q)}\gg\frac{\log y}{(\log(y,t))^{1/\phi(q)}}
\]
and hence 
\[
k\gg\# S\frac{\log y E(X)}{(\log(y,t))^{1/\phi(q)}X}.
\]

Notice that in both Case I and Case II we have completely eliminated any mention of $P_1$ and $P_2$ from our bounds on $k$.  We will proceed by bounding $\# S$ and $\# T$, and from there we can conclude the theorem. We proceed separately for $a\in A_{\pm}$ and $a\not\in A_{\pm}$.

\emph{Case when} $a\in A_{\pm}$. Recall that $i\in S$ if and only if $(i,Q)=1$ and $i\equiv a\pmod{q}$. If $i\in S$ and $a\in A_+$, then since $m\equiv a-1\pmod{q}$, we have $i-m\equiv1\pmod{q}$; and as $m$ is divisible by every factor of $Q$ except $q$ and $q$ does not divide $i-m$, we have $(i-m,Q)=1$. Conversely, if indeed $i-m\equiv1\pmod{q}$ and $(i-m,Q)=1$, we can see by the same reasoning that $(i,Q)=1$ and $i\equiv a\pmod{q}$, so that $i\in S$. That is, setting $j=i-m$, we have
\[
\# S=\#\{j\in[1,yz]:(j,Q)=1,j\equiv1\spmod{q}\},
\]
and, similarly,
\[
\# T =\#\{j\in[1,yz]:(j,Q)=1,j\not\equiv1\spmod{q}\}.
\]
By similar reasoning but setting $j=n-i$, we have the same bijection for $\# S$ and $\# T$ in the case of $A_-$. For convenience, we will refer to these sets also by $S$ and $T$.

Notice that any element of $[1,yz]$ which is a product of primes that are $1\pmod{q}$ lies in $S$, since such an element is $1\pmod{q}$ and none of its prime factors divide $Q$. Thus, by Lemma~\ref{scriptS}, $$\# S\geq\# \mathscr{S}_q(yz)\gg\frac{yz(\log y)^{1/\phi(q)}}{\log y}.$$  

To bound $\# T $ from above, we just need to count those elements of $[1,yz]$ which are relatively prime to $Q$ and not $1\pmod{q}$. But $Q$ is divisible by every prime no greater than $y$ and not $1\pmod{q}$ except for $p_0$; therefore, if $j$ is relatively prime to $Q$, $j$ must be a product of primes that are either greater than $y$, equivalent to $1\pmod{q}$, or equal to $p_0$. Furthermore, $j$ can only be divisible by one prime larger than $y$ since $j\le yz<y^2$, and $j$ cannot consist entirely of prime factors that are $1\pmod{q}$, or it would itself be $1\pmod{q}$. 

We thus narrow $j$ down to the following two sets: numbers $pn$ where $p$ is a prime larger than $y$ and $n$ consists entirely of prime factors that are $1\pmod{q}$, and multiples of $p_0$. The second set is of size $yz/p_0=O(yz/\log y)$, since $p_0\gg\log y$. 

For the first set, we count the elements $pn$ where $p$ belongs to each of $[y+1,2y],[2y+1,4y],[4y+1,8y],\dotsc,[2^{\lceil\log z\rceil-1}y+1,2^{\lceil\log z\rceil}y]$ separately. Now if $p>2^{l-1}y$ then $n<z/2^{l-1}$ (since $pn\le yz$), so $n\in\mathscr{S}_q(z/2^{l-1})$. Using Lemma~\ref{scriptS} to estimate $\mathscr{S}_q(z/2^{l-1})$, we may compute
\begin{align*}
\#\{pn\}&\ll\sum_{l\le\log z}\ \sum_{2^{l-1}y<p\le 2^ly}\ \sum_{n\in \mathscr{S}_q(z/2^{l-1})}1\\
&\ll \sum_{l\le\log z}\ \sum_{2^{l-1}y<p\le 2^ly}\frac{(z/2^{l-1})(\log z)^{1/\phi(q)}}{\log z}\\
&\ll \sum_{l\le\log z}\frac{2^{l-1}y}{\log y}\frac{(z/2^{l-1})(\log z)^{1/\phi(q)}}{\log z}\\
&\ll \frac{yz(\log z)^{1/\phi(q)}}{\log y},
\end{align*}
since the number of primes in $[2^{l-1}y+1,2^ly]$ is $O(2^{l-1}y/\log y)$ by the Prime Number Theorem. Hence, we obtain $$\#T\ll \frac{yz(\log z)^{1/\phi(q)}}{\log y}.$$

\emph{Case when} $a\notin A_{\pm}$. Recall from Equation~\ref{Qdivs} that $Q$ is the product of the primes in the set
\begin{eqnarray*}
\{p\le y\mid p\neq p_0,p\not\equiv1,a\spmod{q}\}&\cup&\{t\le p\le y\mid p\neq p_0,p\equiv1\spmod{q}\}\\
&\cup&\{p\le yz/t\mid p\neq p_0,p\equiv a\spmod{q}\}.
\end{eqnarray*}

To bound $\#S$ from below, we note that an element of $[1,yz]$ of the form $pn$ where $p>yz/t$, $p\equiv a\pmod{q}$ and $n$ is a product of primes which are $1\pmod{q}$ must belong to $S$. Now we use the same strategy as we did for $\#T$ above and count such $pn$ where $p$ belongs to each of $[yz/t+1,2yz/t],[2yz/t+1,4yz/t],\dotsc,[2^{\floor{\log t}-1}yz/t+1,2^{\floor{\log t}}yz/t]$ separately. If $p>2^{l-1}yz/t$ it may be multiplied by any $n\in\mathscr{S}_q(t/2^l)$. Therefore we compute
\begin{eqnarray*}
\#S&\gg&\sum_{l\le\log t}\ \sum_{2^{l-1}yz/t<p\le 2^lyz/t}\ \sum_{n\in\mathscr{S}(t/2^l)}1\\
&\gg&\sum_{l\le\log t}\ \sum_{2^{l-1}yz/t<p\le 2^lyz/t}\frac{(t/2^l)(\log t)^{1/\phi(q)}}{\log t}\\
&\gg&\sum_{l\le\log t}\frac{2^{l-1}yz/t}{\log(yz/t)}\frac{(t/2^l)(\log t)^{1/\phi(q)}}{\log t}\\
&\gg&\frac{yz(\log t)^{1/\phi(q)}}{\log y}.
\end{eqnarray*}
To bound $\#T$ from above, we consider as before any $j\in T$. Suppose $j$ is not a multiple of $p_0$. Since all primes below $\sqrt{yz}\le yz/t$ which are $a\pmod{q}$ divide $Q$, $j$ is not divisible by two $a\pmod{q}$-primes below $y$, nor can it be divisible by both a prime above $y$ and an $a\pmod{q}$-prime below $y$. Now if $j$ is divisible by no prime above $y$, since it is not $a\pmod{q}$, it also cannot be the product of one $a\pmod{q}$-prime and any number of $1\pmod{q}$-primes---that is, it must be entirely a product of $1\pmod{q}$-primes less than $t$. Similarly, if $j=pn$ for some $p>y$, then $n$ must be a product of $1\pmod{q}$-primes less than $z$.

Hence, we narrow $j$ down to the following three sets: products of $1\pmod{q}$-primes less than $t$, numbers $pn$ where $p$ is a prime larger than $y$ and $n\in\mathscr{S}_q(z)$, and multiples of $p_0$. The last two can be bounded with the same calculations as in the case $a\in A_{\pm}$. The first set is trivially contained in all products of primes less than $t$, so can be bounded by $\Psi(yz,t)$. We compute
\[
\frac{\log(yz)}{\log t}\ge\frac{\log y}{\log t}=\frac{4\log y\log\log y}{\log y\log\log\log y}=\frac{4\log\log y}{\log\log\log y}
\]
and applying Lemma~\ref{smooth}, we find
\[
\Psi(yz,t)\le yz(\log t)^2\exp(-4\log\log y+o(\log\log y))\ll\frac{yz(\log t)^2}{(\log y)^4}\ll\frac{yz}{\log y}
\]
so that we can again conclude $\#T\ll \frac{yz(\log z)^{1/\phi(q)}}{\log y}$.

In summary, we have $$\#S\gg\frac{yz(\log(y,t))^{1/\phi(q)}}{\log y}\ \ \ \text{ and }\ \ \ \#T\ll\frac{yz(\log z)^{1/\phi(q)}}{\log y}.$$ We can now conclude the main result. In Case I, we have 
\[
k\gg\frac{\# S}{\# T}\gg \left(\frac{\log(y,t)}{\log z}\right)^{1/\phi(q)}.
\]
Since $\log z=o(\log(y,t))$ as shown earlier, this quantity becomes arbitrarily large.

In Case II, we have 
\[
k\gg\# S\frac{\log y E(X)}{(\log(y,t))^{1/\phi(q)}X}\gg\frac{yzE(X)}{X}.
\]
Thus, since $y\gg\frac{\log X}{D(X)},$
\[
k\gg\frac{zE(X)\log X}{XD(X)}.
\]
Since we chose $\frac{XD(X)}{E(X)\log X}=o(z),$ this quantity becomes arbitrarily large as well. Plugging in the expressions for $y, z,$ and $t$ given earlier in the proof yields the bound for $k$ in the theorem statement, so we are done. 
$\square$

\subsection{Proof of Corollary \ref{beatty}}
Consider the set
\[
\B_\alpha=\{\lfloor \alpha n \rfloor : n\in\mathbb{N}\}
\]
where $\alpha\in\R$ is irrational. In Theorem 5.4 of \cite{Banks}, Banks and Shparlinksi prove that if $\alpha$ is positive, irrational, and of finite type, then there is a positive constant $\kappa>0$ such that if $(q,a)=1$ and $1<q\leq X^\kappa$ non-exceptional, then
\[
\sum_{n\leq X, \lfloor \alpha n\rfloor\equiv a\bmod{q}}\Lambda(\lfloor \alpha n\rfloor)=\frac{1}{\alpha}\sum_{n\leq\lfloor\alpha X\rfloor,n\equiv a\bmod{q}}\Lambda(n)+O(X^{1-\kappa})
\]
where $\Lambda$ is the von Mangoldt function defined as usual by 
$$\Lambda(m)=\begin{cases}\log p&\text{if }m=p^i\\ 0&\text{otherwise.}\end{cases}$$
This implies that if $\alpha>0$ is irrational and of finite type, then
$$\pi_{\A}(X;q,a)\asymp\frac{X}{\phi(q)\log X}$$
(the implied constant depending on $\alpha$) where $\A=\{p\in\B_\alpha : p \mbox{ is prime}\}.$ Taking $E(X)=\frac{X}{\log X}$ and $D(X)$ to be the constant $1/\kappa$, we see that $\A$ is well-distributed, so we can apply Theorem \ref{Main} and obtain the desired result.

\begin{remark}
Almost every real number is of finite type, in particular all of the irrational algebraic numbers. Hence, for example, if $(a,q)=1$, then $\B_{\sqrt{2}}$ contains arbitrarily long strings of consecutive primes all equivalent to $a \bmod{q}.$
\end{remark}

\section{Leitmann Primes}

Leitmann showed in~\cite{Leit} that for a certain class of functions $f$ growing at a rate roughly between $X$ and $X^{\frac{12}{11}-\ve}$, the primes of the form $\floor{f(n)}_{n\in\N}$ are distributed between arithmetic progressions in accordance to a generalization of the Prime Number Theorem.  We wish to apply Theorem~\ref{main} to these primes. One problem that arises is that a growth rate of $X^{\frac{12}{11}-\ve}$ is too large, for the primes in the sequence $\floor{n^{\frac{12}{11}-\ve}}_{n\in\N}$ are too sparse for us to apply the Maier matrix method. 
However, we can still apply Leitmann's work if we assume a few further restrictions on $f$. We will specify these restrictions in Section 4.1.

\subsection{Definitions and Theorem Statement}

We will consider functions of the form $f(x)=xg(x)$ where $g$ is in the set $\G$, to be defined below. Functions in $\G$ grow much slower than any power of $x.$ For example, $g(x)=(\log\log x)^B$ lies in $\G$ for any $B>0.$
\begin{definition}
\label{goodg}
Let $\G$ be the set of all functions $g$ from $[c,\infty)\to [2,\infty)$ where $c\geq 1$ such that $g$ satisfies
\[
xg^{(i)}(x)=g^{(i-1)}(x)(\alpha_i-i+o(1))
\] 
for $i\in\{1,2,3\}$, where $\alpha_1>0$, $\alpha_2\geq 0$, $\alpha_1\neq\alpha_2,$ $\alpha_3\neq 3\alpha_1$, and $2\alpha_1+\alpha_3\neq 3\alpha_2$. We also require that $g$ is increasing, unbounded, and 
 satisfying the inequalities $2g'(x)+xg''(x)>0$ and 
\[
\log g(x)= o\left(\frac{\log\log x\log\log\log\log x}{\log\log\log x}\right).
\]
\end{definition}

\begin{definition}
We define $\Pf=\{\floor{f(n)}:n\in\N\}$.
\end{definition}

Our goal in this section is to prove that if $g\in\G$ and $f(x)=xg(x),$ then the set of primes in $\Pf$ is well-distributed. To achieve this, we simply apply Leitmann's results on the distribution of primes in $\Pf$ in arithmetic progressions. We restate Theorem \ref{leitsimple} here as Theorem \ref{leitprimes}.

\begin{theorem}
\label{leitprimes}
Let $g\in\G$ and $f(x)=xg(x)$. Then for all $q$ and $a$ with $(q,a)=1$, there exist arbitrarily long strings of consecutive primes in $\Pf$ that are all congruent to $a \pmod q$.
\end{theorem}

\subsection{Deduction of Theorem~\ref{leitprimes} from Theorem~\ref{main} and Leitmann's result}\ \\
First, we will show that we can indeed apply the results of Leitmann to functions in $\G$ by showing that if $g\in\G$ and $f(x)=xg(x),$ then $f\in\F$ in the notation of \cite{Leit}.

\begin{lemma}
If $g\in\G$, then $f(x)=xg(x)\in \F$.
\end{lemma}
\begin{proof}
Notice that 
\[
f^{(i)}(x)=ig^{(i-1)}(x)+xg^{(i)}(x),
\] and substituting into $xg^{(i)}(x)=g^{(i-1)}(x)(\alpha_i-i+o(1))$ shows that 
\[
xf^{(i)}(x)=f^{(i-1)}(x)(\alpha_i+o(1)).
\]
Since $g,g'>0$, we have $f'(x)>0$, and since $2g'(x)+xg''(x)>0$, we have $f''(x)>0$. Since $1=o(g(x))$, we have $x=o(f(x))$, and since $g(x)\ll x^{1/11-\eps}$, we have $f(x)\ll x^{12/11-\eps}$. Therefore, $f(x)$ satisfies the conditions on functions in $\F$ in \cite{Leit}.
\end{proof}

Letting $\vp=f^{-1}$, we have the following lower bound on $\vp.$
\begin{lemma}
\label{inverse}
Let $g\in\G$, $f(x)=xg(x)$ and $\vp=f\inv$. Then $\vp(x)\gg\frac{x}{g(x)}.$
\end{lemma}

\begin{proof}
Note that $f(\vp(x))=x$ and $f(\vp(x))=\vp(x)g(\vp(x)),$ so $\vp(x)g(\vp(x))=x.$ For large enough $x,$ since $x=o(f(x)),$ $\vp(x)<x.$ Hence, since $g$ is eventually increasing, $g(\vp(x))<g(x)$ for large enough $x.$ This implies that $\vp(x)=\frac{x}{g(\vp(x))}\gg \frac{x}{g(x)}$.
\end{proof}

For the remainder of this proof we will write $\pi_f(X;q,a)$ in place of $\pi_{\Pf}(X;q,a)$ and $r_f(X;q,a)$ for the remainder quantity
\[
\pi_f(X;q,a)-\frac1{\phi(q)}\int_{f(c)}^X\frac{\vp'(t)}{\log t}dt.
\]
\begin{proof}[Proof of Theorem~\ref{leitprimes}]

Since $\int_{f(c)}^X\frac{\vp'(t)}{\log t}dt$ can be approximated by $\frac{\vp(X)}{\log X}$ by Lemma 3.5 of~\cite{Leit}, we want to show
\[
\pi_f(X;q,a)\asymp \frac{1}{\phi(q)}\frac{\vp(X)}{\log X}
\]
 whenever $q^{D(X)}\leq X$ for some $D$. This would imply that $\Pf$ is well-distributed and Theorem~\ref{main} would then imply the desired result. It suffices to show that $r_f(X;q,a)=o\left(\frac{\vp(X)}{\phi(q)\log X}\right)$ for $q\leq X^{1/D(X)}.$
We define
\[
r(y;q,a)=\pi(y;q,a)-\frac1{\phi(q)}\int_2^y\frac{dt}{\log t}.
\]
By Theorem 1.2 of \cite{Leit}, we have
\begin{align*}
r_f(X;q,a)\ll &\frac{\vp(X)\log X}{\phi(q)\log z}\left(\frac1M+\del\right)+\frac{z}{\phi(q)}+1\\
&+\log X\max_{z\le y\le X}\vp'(y)|r(y;q,a)|+\frac{q^{3/2}M^{5/4}\vp(X)(\log X)^{9/4}}{\del z^{\eps_3}}.
\end{align*}
We can choose $z=\sqrt{X},\del^{-1}=M=\exp(d\sqrt{\log X})$ for some positive constant $d$. Now, since $q$ is a good modulus (as given by a fixed zero-free region), we have by the computation in Lemma 2 of~\cite{Maier} that $|r(y;q,a)|\ll\frac{y}{\phi(q)}\exp(-c_2D)$ whenever $q\le y^{1/D}$, where $c_2$ depends only on the same zero-free region.
 We choose $D(y)=\frac{2+\eps}{c_2}\log\log y$ for some positive $\eps$, giving $|r(y;q,a)|\ll\frac{y}{\phi(q)(\log y)^{2+\eps}}$. In addition, by Lemma 3.5 of~\cite{Leit}, we have $\vp'(y)\ll\frac{\vp(y)}{y}$.

We now compute that for $q\le X^{1/D(X)}$, 
\begin{align*}
\log X\max_{z\le y\le X}\vp'(y)|r(y;q,a)| &\ll \frac{\log X}{\phi(q)}\max_{z\le y\le X}\vp(y)\frac{1}{(\log y)^{2+\eps}}\\
&=\frac{\log X}{\phi(q)}\frac{\vp(X)}{(\log X)^{2+\eps}}=\frac1{\phi(q)}\frac{\vp(X)}{(\log X)^{1+\eps}}\\
&=o\left(\frac{\vp(X)}{\phi(q)\log X}\right).
\end{align*}
The other terms are $o\left(\frac{\vp(X)}{\phi(q)\log X}\right)$ as well, so we have
\[
\#\{p\in\mathscr{A}|p\le X, p\equiv a\bmod q\}\asymp \frac{1}{\phi(q)}E(X)
\]
for non-exceptional $q\le X^{1/D(X)}$. 

Now we just need to check that our choices $D(X)=\frac{2+\eps}{c}\log\log X$, $E(X)=\frac{\vp(X)}{\log X}$ satisfy our growth conditions. But we have 
\[
\log D(X)\ll\log\log\log X=o\left(\frac{\log\log X\log\log\log\log X}{\log\log\log X}\right)
\]
and, by Lemma~\ref{inverse},
\[
\log\left(\frac{X}{\log XE(X)}\right)=\log\left(\frac{X}{\vp(X)}\right)\ll\log g(X)=o\left(\frac{\log\log X\log\log\log\log X}{\log\log\log X}\right)
\]
by the definition of $g$. This completes the proof.
\end{proof}

\end{document}